\documentclass[12pt,a4paper]{article}
\usepackage[utf8]{inputenc}
\usepackage[T1]{fontenc}
\usepackage{amsmath}
\usepackage{amssymb}
\usepackage{amsthm}
\usepackage{graphicx}
\usepackage{mathrsfs}
\usepackage{float}
\newtheorem{proposition}{Proposition}[section]

\newtheorem{theorem}{Theorem}[section]
\newtheorem{lemma}{Lemma}[section]

\title{The Minimal (Edge) Connectivity of Some Graphs of Finite Groups}	 
\author {Siddharth Malviy\footnote{First author, Email:  malviysiddharth@gmail.com} ~and Vipul Kakkar\footnote{Corresponding author, Email: vplkakkar@gmail.com }\\ Department of Mathematics \\
	Central University of Rajasthan \\Ajmer, India}

\date{}
\begin{document}
	\maketitle
	\noindent \textbf{{Abstract.}}
	In this paper, we classify all the finite groups $G$ such that the commuting graph $\Gamma_C(G)$, order-sum graph $\Gamma_{OS}(G)$ and non-inverse graph $\Gamma_{NI}(G)$ are minimally edge connected graphs. We also classify all the finite groups $G$ for that, these graphs are minimally connected. We also classify some groups for that the co-prime graph $\Gamma_{CP}(G)$ has minimal edge connectedness. In final part, we classify all the finite groups $G$ for that co-prime graph $\Gamma_{CP}(G)$ is minimally connected.\\
	
	
	
	\section{Introduction}
	There are several graphs that are connected to finite groups. These graphs provide further insight between the algebraic characteristics of groups and the graph theoretic characteristics of related graphs. New constructions and analyses of graphs with important features can be obtained by applying algebraic approaches. \\
	The commuting graph ${\Gamma_C(G)}$ of a group $G$ is the simple undirected graph with group elements as a vertex set and two elements $x$ and $y$ are adjacent if and only if $xy=yx$ in $G$. The co-prime graph ${\Gamma_{CP}(G)}$ of a group $G$ is the simple undirected graph with group elements as a vertex set and two elements $x$ and $y$ are adjacent if and only if the greatest common divisor of order of elements $x$ and $y$ is $1$ in $G$. The order-sum graph ${\Gamma_{OS}(G)}$ of a group $G$ is the simple undirected graph with group elements as a vertex set and two elements $x$ and $y$ are adjacent if and only if $ o(x)+o(y) > |G|$ in $G$. 	The non-inverse graph ${\Gamma_{NI}(G)}$ of a group $G$ is the simple undirected graph with group elements as a vertex set and two elements $x$ and $y$ are adjacent if and only if $x$ is not the inverse of $y$ in $G$. \\ 
	A finite group $G$ is a non-cyclic group of prime exponent if and only if the power graph $P(G)$ is both non-complete and minimally edge connected (see \cite{prime panda}). A finite group $G$ was also shown to be an elementary abelian $2$-group of rank at least $2$ if and only if $P(G)$ is minimally connected and non-complete.	
	The finite nilpotent group $G$ for which the minimum degree and the vertex connectivity of $P(G)$ are equal was classified by Panda and Krishna \cite{power graph paper}. The	groups have been described by Kumar et al. \cite{enh power graph jk} in a way that provides the vertex connectivity and minimum degree of their corresponding enhanced power graphs are equivalent. For finite non-abelian groups with an exponent-order element, Kumar et al. \cite{superpower graph} investigated Hamiltonian-like features of the superpower graph and established strict bounds for the vertex connectivity. Minimal edge connectivity of enhancred power graph and order supergraph of power graph for finite groups have been explored in \cite{main paper}.\\
	This motivates the present paper, which considers the problem of classification of finite groups such that their commuting graphs, order-sum graphs and non-inverse graphs, respectively, are minimally edge connected. In addition to that, we classify all the finite groups for that these graphs are minimally connected. We also classify some groups for that co-prime graph $\Gamma_{CP}(G)$ has minimal edge connectedness. In final part, we classify all the finite groups $G$ for that co-prime graph $\Gamma_{CP}(G)$ is minimally connected. Throughout this paper, $G$ is a finite group and $e$ is the identity element of $G$.
	
	\section{Preliminiaries}
	We now briefly go over some of the terminology and notations we use in the paper. The collection of vertices and edges of a graph $\Gamma$ are denoted by $V (\Gamma)$ and $E(\Gamma)$, respectively. A graph $\Gamma$ is a pair $\Gamma = (V, E)$. If there is an edge $\{uv\}$ between $u$ and $v$, then two distinct vertices $u$ and $v$ are adjacent, indicated by $u \sim v$. When two edges share endpoints, they are said to be multiple edges. If $u = v$, an edge $\{uv\}$ is referred to be a loop. Graphs without loops or many edges are called simple graphs. The neighborhood of $x$ is the set $N(x)$ of all the vertices that are adjacent to the vertex $x$ in $\Gamma$. We additionally indicate $N[x] = N(x) \cup \{x\}$. A graph $\Gamma'$\ such that $V (\Gamma') \subseteq V(\Gamma)$ and $E(\Gamma') \subseteq E(\Gamma)$ is called a subgraph of the graph $\Gamma$. A graph $\Gamma$ is considered complete if any two different vertices are adjacent to each other. From vertex $u$ to vertex $w$, a walk $\lambda$ in $\Gamma$ is a sequence of vertices $u = v_1, v_2, \hdots, v_m = w (m > 1)$ such that $v_i \sim v_{i+1}$ for every $i \in {1, 2, \hdots, m-1}$). If no vertex is repeated, a walk is considered a path. If every pair of vertices in a graph $\Gamma$ has a path in $\Gamma$, then the graph is connected, otherwise $\Gamma$ is not connected.

	The degree of $u$ is the number of vertices that are adjacent to it; it is represented by the notation $deg(u)$. If every vertex in a graph has the same degree, the graph is said to be regular. If a vertex $u$ is adjacent to every other vertex in a graph $\Gamma$, it is considered a dominant vertex. The diameter of $\Gamma$, represented by $diam(\Gamma)$, is the maximum distance between two vertices of a connected graph $\Gamma$. A tree is a connected graph with no cycle. 
	
	The graph $\Gamma$ has a minimum degree $\delta(\Gamma)$, which is the smallest degree of all its vertices. In a connected graph $\Gamma$, a vertex (edge) cut-set is a set $S$ of vertices (edges) such that, after deleting the set $S$, the remaining subgraph $\Gamma-S$ is disconnected or has just one vertex. The smallest size of a vertex (edge) cut-set is the vertex connectivity $\kappa(\Gamma)$(edge connectivity $\kappa'(\Gamma))$ of a connected graph $\Gamma$. If, for any edge $\epsilon$ of graph $\Gamma$, $\kappa'(\Gamma- \epsilon) = \kappa'(\Gamma)-1$, then the graph is considered minimally edge connected. A graph is considered minimally connected if $\kappa(\Gamma-\epsilon)= \kappa(\Gamma)-1$ for any edge $\epsilon$ of graph $\Gamma$. There are two statements that establish relationships among edge connectivity, vertex connectivity, and the minimum degree of a graph.

	\begin{proposition}
		(\cite[Theorem 6]{diameter 2}). If the diameter of any graph is at most 2, then its edge connectivity and minimum
		degree are equal.
	\end{proposition}
	
	\begin{proposition}
		(\cite[Theorem 4.1.9]{west book}) If $\Gamma$ is a simple graph, then $\kappa(\Gamma) \leq \kappa'(\Gamma) \leq \delta(\Gamma)$.
	\end{proposition}
	
	
	Throughout this paper, $\mathbb{Z}_n$ denotes the cyclic group of order $n$. The centralizer of an element $x$ in the group $G$ is denoted by $C_G(x)$. For a positive integer $n$, $\phi(n)$ denotes the Euler’s totient function of $n$. The order of an element $x \in G$ is denoted by $o(x)$. If $G$ contains an element whose order is equal to the exponent $exp(G)$ of $G$, then $G$ is called the group of full exponent. If $|G| = p^n$ for some prime $p$, then $G$ is called a $p$-group. An EPPO-group is a finite group such that the order of each element is a power of a prime number. It should be noted that all $p$-groups are EPPO-groups. The converse need not be true, though. For instance, the dihedral group $D_{2n}$, where $n = p^\alpha$ for some odd prime $p$, and the symmetric group $S_3$ are EPPO-groups but not $p$-groups.

	
	\section{Minimal (Edge) Connectivity}
	
	The main results of the mansuscript are presented in this section. We characterize all the finite groups $G$ such that the order-sum graph $\Gamma_{OS}(G)$ (Theorem \ref{thm order sum graph  minimal edge}), the non-inverse graph $\Gamma_{NI}(G)$ (Theorem \ref{thm non inverse minimal edge}) and the commuting graph $\Gamma_C(G)$ (Theorem \ref{thm commuting minimal edge}) are minimally edge connected graphs. We also classify all the finite groups $G$ for that, these graphs are minimally connected ( Theorem \ref{cor order-sum minimal connected}, Theorem \ref{cor non-inverse minimal connected}, Theorem \ref{thm commuting minimal connected}) respectively. We also classify some groups for that co-prime graph $\Gamma_{CP}(G)$ has minimal edge connectivity. In final part, we classify all the finite groups for that co-prime graph $\Gamma_{CP}(G)$ is minimal connected.

	\begin{lemma} \label{lemma 3.1}
		If a graph is a complete graph or a star graph, then it is minimally edge connected and minimally connected.
	\end{lemma}
	
	\begin{lemma} \label{lemma 3.2}
		The commuting graph $\Gamma_C(G)$ of a group $G$ is complete if and only if $G$ is abelian.
	\end{lemma}
	
	\begin{lemma} \label{lemma completeness coprimes}
		The co-prime graph $\Gamma_{CP}(G)$ of a group $G$ is complete if and only if order $|G|$ of group $G \leq 2$ .
	\end{lemma}
	
	\begin{lemma} \label{lemma 3.4} \cite[Theorem 8]{order sum graph}
		The order-sum graph $\Gamma_{OS}(G)$ of a group $G$ is complete if and only if $G$ is cyclic group of prime order.
	\end{lemma}
	
	\begin{lemma} \label{lemma 3.5}
		The non-inverse graph $\Gamma_{NI}(G)$ of a group $G$ is complete if and only if all the elements in $G$ is self-inverse.
	\end{lemma}
	
	\begin{lemma}
		\cite[Theorem 3.6]{non-inverse graph} For a non-inverse graph $\Gamma_{NI}(G)$ associated with a group $G$, the vertex-connectivity $\kappa(\Gamma_{NI})$ and edge-connectivity $\kappa(\Gamma_{NI})$ are always equal.
	\end{lemma}

\begin{proposition} \label{minimal edge pro}
	\cite[Theorem 2.1]{main paper}	Let $\Gamma$ be a non-complete graph with a dominating vertex $x$. Then $\Gamma$ is minimally edge connected if and only if $x$ is the only dominating vertex in $\Gamma$ and $\Gamma-\{x\}$ is a regular graph.
\end{proposition}

\begin{proposition}
	\cite{order sum graph}	If $G$ is non-cyclic group, then order-sum graph $\Gamma_{OS}(G)$ is null graph.
\end{proposition}

\begin{theorem} \label{thm order sum graph  minimal edge}
	Let $G$ be a finite cyclic group. Then the order-sum graph $\Gamma_{OS}(G)$ is minimally edge connected if and only if $G$ is the group of order a prime.
\end{theorem}

\begin{proof}
	If $G$ is of order prime, then by Lemma \ref{lemma 3.4}, $\Gamma_{OS}(G)$ is complete and hence minimal edge connected (by Lemma \ref{lemma 3.1}). Now, suppose $G$ is a cyclic group with non-prime order. Then $\Gamma_C(G)$ is non-complete graph.
	Note that $x \in G$ is the dominating vertex if $o(x)=o(G)$. By Proposition \ref*{minimal edge pro}, $x$ should be the only dominating vertex. This implies that $\phi(n)=1$. Thus $|G| \leq 2$. This is a contradiction.		
\end{proof}

\begin{theorem} \label{thm non inverse minimal edge}
	Let $G$ be a finite group. Then the non-inverse graph $\Gamma_{NI}(G)$ is minimally edge connected if and only if all the non-identity elements are either self-inverse or not self-inverse.
\end{theorem}

\begin{proof}
	If $G$ is group with all the non-identity elements are self-inverse , then by Lemma \ref{lemma 3.5}, $\Gamma_{NI}(G)$ is a complete graph and hence minimal edge connected (by Lemma \ref{lemma 3.1}).  If all the non-identity elements are not self-inverse, then $\Gamma_{NI}(G)$ is a star graph, hence minimal edge connected. Let $x$ be a non-identity self inverse element and $y$ be a non-identity  not self inverse element. Then  $\Gamma_{NI}(G)$ is non-complete graph and $ deg(x) \neq deg(y)$. By Proposition \ref*{minimal edge pro}, it is not minimally edge connected. It completes the proof.
\end{proof}

	\begin{theorem} \label{thm commuting minimal edge}
	Let $G$ be a finite group. Then the commuting graph $\Gamma_C(G)$ is minimally edge connected if and only if $G$ is abelian.
\end{theorem}

\begin{proof}
	If $G$ is abelian, then by Lemma \ref{lemma 3.2}, $\Gamma_C(G)$ is complete and hence minimal edge connected (by Lemma \ref{lemma 3.1}). Now, suppose $G$ is non-abelian. Then $\Gamma_C(G)$ is non-complete graph. From Proposition \ref{minimal edge pro}, the center of the group $G$ should consists only identity element. For $\Gamma_C(G)-\{e\}$ to be regular, the size of centralizer of all the non-central elements has to be equal. Let there be $m$ centralizers of non-central elements each of size $k$. Now, class equation of group $G$,  $$|G|= |Z(G)| + \sum_{x \in G \backslash Z(G)} \frac{|G|}{|C_G(x)|},$$
	implies that
	$$ \frac{|G|-1}{|G|}= \frac{m}{k},$$ which is not possible. It completes the proof.
\end{proof}

\begin{theorem} \label{thm commuting minimal connected}
	Let $G$ be a finite group. Then the commuting graph $\Gamma_{C}(G)$ is minimally connected if and only if $G$ is abelian.
\end{theorem}

\begin{proof}
	Assume that $\Gamma_{C}(G)$ has minimal connectivity. Let $G$ be a non-abelian group. Then there exists an element $a$ such that $o(a) \geq 3$. Suppose $\epsilon=\{aa^{-1}\}$, that is $a \sim a^{-1}$. Let $P$ be a vertex cut-set of minimal size in the graph $\Gamma_{C}(G)-\epsilon$. Consequently, $\kappa(\Gamma_{C}(G)-\epsilon) = |P|$. Since $\Gamma_{C}(G)$ is minimally connected, we get $\kappa(\Gamma_{C}(G))= |P|+1$.  Thus, there are two elements $b_1$ and $b_2$ of $G$ in $(\Gamma_{C}(G)-\epsilon)-P$ such that there is no path connecting them. We have two cases:\\
	Case-1: For at least one $i$, $b_i \notin \{ a, a^{-1}\}$. The graph $\kappa(\Gamma_{C}(G))-P$ is connected since $\kappa(\Gamma_{C}(G))= |P|+1$. Consequently, a path $P_1: b_1 \sim a_1 \sim a_2 \sim \ldots \sim a_k \sim a \sim a^{-1} \sim \ldots \sim b_2$ (including the edge $\epsilon$) between $b_1$ and $b_2$ exists in $\kappa(\Gamma_{C}(G))-P$. Since $a_k$ commutes with $a$, $a_k$ commutes with $a^{-1}$, it implies that $a_k \sim a^{-1}$. In $(\Gamma_{C}(G)-\epsilon)-P$, we have an alternative path $P_2: b_1 \sim a_1 \sim a_2 \sim \ldots \sim a_k \sim a^{-1} \sim \ldots \sim a_t \sim b_2$, which is not possible.\\
	Case-2: For $1 \leq i \leq 2$, $b_i \in \{a,a^{-1}\}$. First assume that $G = P \cup \{a, a^{-1}\}$. Then $|P| = |G|- 2$. Consequently, $\kappa(\Gamma_{C}(G))=|G|-1$. As a result, $\Gamma_{C}(G)$ is complete. This implies that $G$ is abelian. That is a contradiction. Now we can assume that $G \ne P \cup \{a, a^{-1}\}$. If $a \sim b$, then $b \sim a^{-1}$ for some $b \in G \backslash P$. This implies that there is a path between $a$ and $a^{-1}$. This is a contradiction.  Hence $a$ and $a^{-1}$ are isolated vertices in $(\Gamma_{C}(G)-\epsilon)-S$ in this instance. Therefore, $a$ and $a^{-1}$ are not connected to others in $\Gamma_{C}(G)-P$. Hence, the graph $\Gamma_{C}(G)-P$ is disconnected. This implies that $\kappa(\Gamma_{C}(G)) \leq |P|$, which leads to a contradiction.
	Hence, $G$ is abelian. \\
	Conversely, if $G$ is abelian, then by Lemma \ref{lemma 3.2}, $\Gamma_{C}(G)$ is complete graph and hence minimally connected.
\end{proof}

\noindent Alternate proof: A graph $\Gamma$ is minimally connected if and only if it is tree. Note that for a tree $\Gamma$, $\kappa'(\Gamma)=1$ because tree is an acyclic structure. Removing any one edge make it disconnected. Hence $\kappa'(\Gamma-\epsilon)=0$ for an edge $\epsilon$. Hence $\Gamma$ is minimally edge connected. Therefore, if $\Gamma$ is not minimally edge connected then it is also not a minimally connected.\\
If $G$ is not-abelian then $\Gamma_{C}(G)$ is not minimally edge connected. Hence it is also not minimally connected. Also, if $G$ is abelian then $\Gamma_{C}(G)$ is complete hence minimally connected.\\

\noindent By the similar argument as above, we can find when the order-sum graph graph $\Gamma_{OS}(G)$ and the non-inverse graph $\Gamma_{NI}(G)$ are minimally connected:

\begin{theorem} \label{cor order-sum minimal connected}
	Let $G$ be a finite group. Then the order-sum graph graph $\Gamma_{OS}(G)$ is minimally connected if and only if $G$ is a group of prime power order.
\end{theorem}

\begin{theorem} \label{cor non-inverse minimal connected}
	Let $G$ be a finite group. Then the non-inverse graph $\Gamma_{NI}(G)$ is minimally connected if and only if all the non-identity elements are either self-inverse or not self-inverse.
\end{theorem}

\begin{proposition}{\label{sg full exponent}}
	Let $G$ be a finite group of full exponent. Then the co-prime graph $\Gamma_{CP}(G)$  is minimally edge connected if and only if $G$ is a $p$-group.
\end{proposition}

The following theorem describes minimal edge-connectivity of the co-prime graph $\Gamma_{CP}(G)$ of the group $G$ of even order.

\begin{theorem}
	Let $G$ be a group of even order which is not a $p$-group. Then the co-prime graph $\Gamma_{CP}(G)$  is not minimally edge connected.
\end{theorem}
\begin{proof}
	Let $G$ be a group of even order which is not a $p$-group. The only dominating vertex of $\Gamma_{CP}(G)$ is the identity element. Let $a,b \in G$ such that $o(a)=2$ and $o(b) = p$, where $p$ is an odd prime. The group $G$ has odd number of elements of order two. For $b_1 \in G$, if $a \sim b_1$ such that $o(b_1) > 2$ (but not even), then $a \sim b_1^{-1}$. Consequently, $deg(a)$ is even in $\Gamma_{CP}(G)\backslash \{e\}$. Now, if $ b \sim c$ for $c \neq e \in G$ such that $o(c) > 2$, then $b \sim c^{-1}$. Also $b \sim x$ (all the elements of order 2). Thus $deg(y)$ is odd in $\Gamma_{CP}(G)\backslash \{e\}$ . It follows that $\Gamma_{CP}(G)$  is not regular. By Proposition \ref{minimal edge pro}, $\Gamma_{CP}(G)$  is not minimally edge connected. 
\end{proof}


\begin{theorem} \label{thm coprime minimal connected} 
	Let $G$ be a finite group. Then the co-prime graph $\Gamma_{CP}(G)$ is minimally connected if and only if $G$ is $p$-group.
\end{theorem}

\begin{proof}
	Assume that $\Gamma_{C}(G)$ has minimal connectivity. Let $G$ be not a $p$-group. Then there exist two elements $a$ and $b$ such that $o(a)=p_1$ and $o(b)=p_2$ where $p_1$ and $p_2$ are distinct primes with $ p_1 > p_2$ . Suppose $\epsilon=\{ab\}$. In the graph $\Gamma_{CP}(G)-\epsilon$, let $P$ be a vertex cut-set of minimal size. Consequently, $\kappa(\Gamma_{CP}(G)-\epsilon) = |P|$. Since $\Gamma_{CP}(G)$ is minimum connected, we get $\kappa(\Gamma_{CP}(G))= |P|+1$ . Thus, in $(\Gamma_{CP}(G)-\epsilon)-P$, there are two elements $b_1$ and $b_2$ of $G$ such that there is no path connecting them. We have two cases:\\
	Case-1: For at least one $i$, $b_i \notin \{ a, b\}$. The graph $\kappa(\Gamma_{C}(G))-P$ is connected since $\kappa(\Gamma_{C}(G))= |P|+1$. Consequently, a path $P_1: b_1 \sim a_1 \sim a_2 \sim \ldots \sim a_k \sim a \sim b \sim \ldots \sim b_2$ (including the edge $\epsilon$) between $b_1$ and $b_2$ exists in $\kappa(\Gamma_{C}(G))-P$. We obtain $a_k \sim a^{-1}$ since $a_k \sim a$. In $(\Gamma_{CP}(G)-\epsilon)-P$, we have an alternative path $P_2: b_1 \sim a_1 \sim a_2 \sim \ldots \sim a_k \sim a^{-1} \sim b \sim \ldots \sim a_t \sim b_2$, which is not possible.\\
	Case-2: For $1 \leq i \leq 2$, $b_i \in \{a, b\}$.  First assume that $G = P \cup \{a, b\}$. Then $|P| = |G|- 2$. Consequently, $\kappa(\Gamma_{CP}(G))=|G|-1$. As a result, $\Gamma_{CP}(G)$ is complete. Then by Lemma \ref{lemma completeness coprimes}, $|G| \leq2$. That is a contradiction. \\
	Now we can assume that $G \ne P \cup \{a, b\}$. If there exist an element $c \in G \backslash P$ such that $o(c)= p_3$, where $p_3$ is a prime distinct from $p_1$ and $p_2$. If $ c \sim a$, then $c \sim b$. That is not possible, otherwise there is a path between $a$ and $b$ in $(\Gamma_{CP}(G)-\epsilon)-P$, that is a contradiction.  Hence $G$ is divisible by only two primes.\\
	If $G$ is divisible by only two primes, then for non-EPPO-group there exist an element $w$ such that $o(w)=p_1 p_2$. Hence $w$ is isolated vertex in $\Gamma_{CP}(G) - P$, which makes $\Gamma_{CP}(G) - P$ disconnected, that is a contradiction. \\
	Now, for EPPO-group, if $\epsilon = \{ex\}$ where $x\in P$, then $\kappa(\Gamma_{CP}(G)) \neq \kappa(\Gamma_{CP}(G)-\epsilon)$. This implies that it is not minimally connected.
	Hence $G$ is $p$-group.\\
	Conversely, if $G$ is $p$-group, then $\Gamma_{CP}(G)$ is star graph and hence minimally connected (by Lemma \ref{lemma 3.1}). It completes the proof.
\end{proof}

	\section*{Declarations}
	
	\noindent \textbf{Acknowledgement}: The first author is supported by junior research fellowship of CSIR, India.\\
	
	\noindent\textbf{Conflicts of interest}: There is no conflict of interest regarding the publishing of this paper.
	

\end{document}